\documentclass[11pt]{article}
\usepackage{amsfonts,amsmath,amsthm,amssymb}

\theoremstyle{plain}

\newtheorem*{remark}{Remark}
\newtheorem{proposition}{Proposition}[section]
\newtheorem{example}{Example}
\newtheorem{definition}{Definition}[section]

\begin{document}

\title{\textbf{Weingarten surfaces with moving frames - a tribute to S.S. Chern and C.L. Terng - and a duality result}}
\author
{Magdalena Toda\\
Department of Mathematics and Statistics\\
Texas Tech University\\
Lubbock, TX  79409-1042}
\date{}
\maketitle
\numberwithin{equation}{section}
\thispagestyle{empty}
\begin{abstract}
The techniques used in this paper are based on the exterior calculus of Maurer-Cartan forms, and Weingarten surfaces are used to illustrate the methods that apply to quadratic exterior equations with constant coefficients. Isothermic {\it surfaces of constant astigmatism} (non-linear Weingarten surfaces whose difference of principal curvatures is a constant) are shown to represent dual surfaces of isothermic surfaces which satisfy the relation $H + \alpha K = 0$.
\end{abstract}


\newpage
\section{Introduction}

Several decades ago, it was discovered that some important PDEs represent the Gauss-Codazzi equations of some types of surfaces studied in differential geometry. Transformations that generate new surfaces from given ones became a central research topic in geometry. Chern and his collaborators (cf. [4][6][7][8][9]) published numerous papers in the field of geometric transformations with moving frame methods. Most of these transformations satisfy a certain \emph{nonlinear superposition principle} in the theory of integrable systems.

In this work, we present cases of quadratic exterior equations of constant coefficients, especially the structure equations of linear Weingarten surfaces. We are also discussing new results related to the integrability of Weingarten surfaces. We prove that {\it isothermic surfaces of constant astigmatism} which are presented in [10] actually represent the dual surfaces (Christoffel transforms) of  isothermic linear Weingarten surfaces whose mean and Gauss curvatures satisfy the relation $H + \alpha K = 0$. \\

\section{Moving frames and Cartan's structure equations}
Cartan's method of moving frames has been fruitfully used in surface theory in $E^3$ over
the past few decades.
Cartan observed that the Gauss-Codazzi-Mainardi-Peterson equations are best derived
from the integrability conditions satisfied by the so called Maurer-Cartan forms of
the Euclidean motion
group.  A frame is a collection $\{x, e_1, e_2, e_3\}$ where $x$ is a point in $E^3$ and
$\{e_1, e_2, e_3\}$ a set of orthonormal vectors. The set of all frames represents a
6-dimensional manifold. One can write
\begin{equation}
\begin{cases}
dx = \omega_1 e_1 + \omega_2 e_2 + \omega_3 e_3,\\
d e_1 = \omega_{11} e_1 + \omega_{12} e_2 + \omega_{13} e_3,\\
de_2 = \omega_{21} e_1 + \omega_{22} e_2 + \omega_{23} e_3,\\
de_3 = \omega_{31} e_1 + \omega_{32} e_2 + \omega_{33} e_3.
\end{cases}
\end{equation}
The differential 1-forms $\omega_i$ and $\omega_{ij} = -\omega_{ij}$ defined on the
frame space $\mathcal{F}$ are infinitesimal components of frame displacement.  They
can also be viewed as Maurer-Cartan forms (right invariant 1-forms on the Euclidean
motion group).  Here we adopt the convention that the group transforms $E^3$ by right
multiplications.  Using $d^2 x = d^2 e_i =0$,
we obtain the integrability conditions
\begin{equation}
\begin{cases}
d\omega_i = \omega_j \wedge \omega_{ji},\\
d \omega_{ij} = \omega_{i\kappa} \wedge \omega_{\kappa j},
\end{cases}
\end{equation}
called {\it the structure equations} of the group $G$.  Here we use the
summation convention.  Since the matrix $\omega_{ij}$ is antisymmetric, we have 6
linearly independent 1-forms $\{\omega_1, \omega_2, \omega_3, \omega_{12},
\omega_{13}, \omega_{23}\}$ which form a basis of all the right invariant 1-forms on
$\mathcal{F}$.

For a surface $M$ in $E^3$ given by $x = x(u,v)$, all the frames $\{x, e_1, e_2,
e_3\}$ satisfying $x \in M$ are called the zeroth order frames of the surface $M$,
which form a 5-dimensional submanifold $\mathcal{F}_0 \subset \mathcal{F}$. Via
restriction to this submanifold, the three linear independent 1-forms $\omega_1,
\omega_2, \omega_3$ will satisfy a linear relation $a_1 \omega_1 + a_2 \omega_2 + a_3
\omega_3 =0$, which is essentially the equation of the tangent plane of $M$ at $x$
relative to the frame $\{x, e_1, e_2, e_3\}$.  The coefficients $a_1, a_2, a_3$ vary
from frame to frame. If the frame is such that $e_1, e_2$ span the tangent plane of
$M$ at $x$, the above linear relation takes the form
\begin{equation}
\omega_3 =0.
\end{equation}
All such frames are called the first order frames of $M$ which form a 3-dimensional
submanifold $\mathcal{F}_1 \subset \mathcal{F}_0$.  Restricting all Maurer-Cartan
forms to $\mathcal{F}_1$, the structure equations (2.3) hold true, and we have the
additional equation $\omega_3=0$ and its consequence $d \omega_3 =0$. Equation (2.3)
now becomes
\begin{equation}
\begin{cases}
d \omega_1 = - \omega_2 \wedge \omega_{12},\; d \omega_2 = \omega_1 \wedge \omega_{12} ,\\
d \omega_{12} = -\omega_{13} \wedge \omega_{23},\\
d \omega_{13} = \omega_{12} \wedge \omega_{23},\; d \omega_{23} = \omega_{13} \wedge \omega_{12},\\
\omega_3 =0, \;\omega_1 \wedge \omega_{13} +  \omega_2 \wedge \omega_{23} =0.
\end{cases}
\end{equation}
These are Cartan's structure equations for the first order frames of surfaces in
$E^3$.  By Cartan's lemma, $\omega_1 \wedge \omega_{13} + \omega_2 \wedge \omega_{23}
=0$ implies
\begin{equation}
\begin{cases}
\omega_{13} = h_{11} \omega_1 + h_{12} \omega_2,\\
\omega_{23} = h_{12} \omega_1 + h_{22} \omega_2.
\end{cases}
\end{equation}
The first and second fundamental forms of the surface $M$ are given by
\begin{equation}
\begin{cases}
I = dx \cdot dx = \omega^2_1 + \omega^2_2,\\
II = -dx \cdot d e_3 = \omega_1 \omega_{13} + \omega_2 \omega_{23} = h_{11} \omega^2_1
+ 2 h_{12} \omega_1 \omega_2 + h_{22} \omega^2_2.
\end{cases}
\end{equation}
The two form $\omega_1 \wedge \omega_2$ is the area element of the surface.  The
Gaussian curvature $K = h_{11} h_{22} - h^2_{12}$ and mean curvature $H = (h_{11} +
h_{22})/2$ are also given by
\begin{equation}
\begin{cases}
\omega_{13} \wedge \omega_{23} = K \omega_1 \wedge \omega_2,\\
\omega_1 \wedge \omega_{23} - \omega_2 \wedge \omega_{13} = 2 H \omega_1 \wedge
\omega_2.
\end{cases}
\end{equation}
It is remarkable that the first 5 equations of (2.5) are consequences of the last two
equations.  Therefore (2.5) is essentially a right invariant differential system on
$\mathcal{F}$ defined by $\omega_3 = \omega_1 \wedge \omega_{13} + \omega_2 \wedge
\omega_{23} =0$.

The following propositions were used by Cartan in [2], [3]; they represent direct consequences of the Frobenius theorem.

\begin{proposition}
If $M$ is a surface in $E^3$ given by $x = x (u,v)$, then the 3-dimensional space
$\mathcal{F}_1$ of $M$'s first order frames given by $F = F(u,v,\phi)$ is an integral
manifold of (2.5).  In particular, any first order frame field $F = F(u,v,\phi(u,v))$
is also an integral manifold.

Conversely, if $F=F(u,v)$ is an integral manifold of (2.5) satisfying $\omega_1 \wedge
\omega_2 \neq 0$, then the corresponding surface $M$ given by $x = x(u,v)$ is an
immersed surface of $E^3$.  Moreover $F=F(u,v)$ is a first order frame field of $M$.
\end{proposition}

A second view of the structure equations (2.5) is summarized in the following
\begin{proposition}
If $M$ is a surface of $E^3$ given by $x = x(u,v)$, then along any first order frame
field $F=F(u,v)$, the Maurer-Cartan forms $\omega_i$ and $\omega_{ij}$ can be viewed
as forms on $M$ which still satisfy (2.5).  Conversely, given a set of 1-forms
$\omega_i = p_i (u,v)du+q_i (u,v)dv$ and $\omega_{ij}=p_{ij}(u,v)du+q_{ij}(u,v)dv$
satisfying (2.5) and $\omega_1 \wedge \omega_2 \neq 0$, one can reconstruct the
immersion $x=x(u,v)$ into $E^3$ uniquely upto Euclidean motion.
\end{proposition}

The previously-mentioned results have a series of very important and well-known consequences.
Surfaces of constant Gaussian curvature $K$ may be viewed as integral manifolds of the constant coefficient
differential system
\begin{equation}
\begin{cases}
d \omega_1 = \omega_{12} \wedge \omega_2, \;d \omega_2 = \omega_1 \wedge \omega_{12},\\
d \omega_{12} = -\omega_{13} \wedge \omega_{23},\\
d \omega_{13} = \omega_{12} \wedge \omega_{23},\; d \omega_{23} = \omega_{13} \wedge \omega_{12},\\
\omega_3 =0,\; \omega_1 \wedge \omega_{13} + \omega_2 \wedge \omega_{23} =0,\\
\omega_{13} \wedge \omega_{23} = K \omega_1 \wedge \omega_2,\; \omega_1 \wedge
\omega_2 \neq 0,
\end{cases}
\end{equation}
on $\mathcal{F}$.  The additional condition $\omega_1 \wedge \omega_2 \neq 0$ means we
require the surface to be immersed in $E^3$.  When $K < 0$, the system (2.9) is equivalent
to the sine-Gordon equation
\begin{equation}
\psi_{uv} = \sin \psi.
\end{equation}
Here $(u,v)$ is a special asymptotic coordinate system on the surface, $\psi = \psi
(u,v)$ the angle formed by the two asymptotic lines at $(u,v)$.

Surfaces of constant mean curvature $H$ are obtained by replacing the last equation of (2.9) with $\omega_1 \wedge \omega_{23} - \omega_2
\wedge \omega_{13} = 2H \omega_1 \wedge \omega_2$, where the mean curvature $H$ is a constant.  This equation is equivalent to the
sinh-Gordon equation.
\begin{equation}
\psi_{xx} + \psi_{yy} + \sinh \psi =0.
\end{equation}
Sine-Gordon equation and sinh-Gordon equation are well known integrable systems.  We
will study them mainly in the form of the above constant coefficient exterior
equations.

Remark that surfaces of constant Gaussian curvature and surfaces of constant mean curvature are two very important examples of linear Weingarten surfaces, and therefore well worth mentioning here.

\section{Some examples of geometric transformations}

I owe this paragraph to the dedication of X. Mo in presenting geometric transformations
in a general survey format. This section represent a tribute to the works [4]-[9] of
Chern, which served as a main reference.

First, let us recall the classical transformation between pseudospherical surfaces
found by Bianchi and B\"{a}cklund.

\begin{proposition}
(B\"{a}cklund theorem) Let $M$ and $M^\ast$ be two surfaces in $E^3$ with a one to one correspondence between $P\in M$ and $P^\ast \in M^\ast$ such that:\\
\indent (i) $PP^\ast$ is tangent to $M$ and $M^\ast$ at $P$ and $P^\ast$ respectively,
(ii) the distance $r = ||PP^\ast||$ and the angle $\tau$ between the normal directions
of $M$ and $M^\ast$ at $P$ and $P^\ast$ are constants.

Then both $M$ and $M^\ast$ have constant negative Gaussian curvature $K = -\sin^2
\tau/r^2$.
\end{proposition}
\begin{proposition}
(B\"{a}cklund integrability theorem) Consider the constants $r>0$ and $\tau$ and assume $M$
is a surface of constant negative Gaussian curvature $-\sin^2 \tau/r^2$.  Then for any
point $P_0 \in M$ and $P^\ast_0$ in space such that $P_0 P^\ast_0$ is tangent to $M$
at $P_0$ but not in the principle direction, there exists a unique surface $M^\ast$
tangent to $P_0 P^\ast_0$ at $P^\ast_0$, which also satisfies the other conditions of
proposition 3.1.
\end{proposition}

\begin{proof}
Chern and Terng's elegant derivation of these classical results in [9] is recalled below.
The condition of proposition 3.1 means that we can choose first order frames fields
$F=\{x, e_1, e_2, e_3\}$ for $M$ and $F^\ast = \{x, e^\ast_1, e^\ast_2, e^\ast_3\}$
for $M^\ast$ such that
\begin{equation}
\begin{cases}
x^\ast = x + r e_1,\\
e^\ast_1 = e_1,\\
e^\ast_2 = \cos \tau e_2 + \sin \tau e_3,\\
e^\ast_3 = -\sin \tau e_2 + \cos \tau e_3,
\end{cases}
\end{equation}
holds at all pairs of corresponding points.  This implies
\begin{equation}
\omega^\ast_3 = dx^\ast \cdot e^\ast_3 = -\sin \tau \omega_2 - r \sin \tau \omega_{12}
+ r \cos \tau \omega_{13}.
\end{equation}
From the structure equation of $M^\ast$ we have $\omega^\ast_3 =0$.  Therefore
\begin{equation}
\omega_2 + r \omega_{12} - r \cot \tau \omega_{13} =0
\end{equation}
holds along the frames fields $F$.  Consider $\alpha = \omega_2 + r \omega_{12} - r
\cot \tau \omega_{13}$ as a form on the space of all first order frames of $M$ (not
restricted to the frame fields chosen above).  Using the structure equation (2.5) of
$M$ we can verify that
\begin{equation}
d \alpha = \beta\, \pmod{\alpha}
\end{equation}
where
\begin{equation}
\beta = -[\omega_1 \wedge \omega_2 + (r^2 /\sin^2 \tau) \omega_{13} \wedge
\omega_{23}]/r.
\end{equation}
Restricting to the frame field $F$, we have $\alpha=0$ and therefore $d \alpha =0$
which gives $\beta=0$.  By $\omega_{13} \wedge \omega_{23} = K \omega_1 \wedge
\omega_2$, we obtain $K = -\sin^2 \tau/r^2$.  This proves the proposition 3.1.  The
condition of proposition 3.2 is $K = -\sin^2 \tau/r^2$, namely $\beta=0$.  The
conclusion follows from (3.3) and Frobenius theorem.

\end{proof}

$M^\ast$ is called the B\"{a}cklund transformation of $M$.  In terms of the function
$\psi = \psi (u,v)$ which satisfies the sine-Gordon equation (2.10), B\"{a}cklund
transformation given above takes the analytic form
\begin{equation}
\begin{cases}
(\frac{\psi + \psi^\ast}{2})_u = \frac{1}{\lambda} \sin \frac{\psi - \psi^\ast}{2},\\
(\frac{\psi - \psi^\ast}{2})_v = \lambda \sin \frac{\psi + \psi^\ast}{2}.
\end{cases}
\end{equation}

One of the most important property of B\"{a}cklund transformation is the following
\begin{proposition}
(Bianchi permutability theorem) If a surface $M_0$ has two transformations $M_1$ and
$M_2$ with parameters $(r_1, \tau_1)$ and $(r_2, \tau_2)$ respectively, then there can
be found a fourth surface $M_3$ which is a transformation of both $M_1$ and $M_2$ with
parameters $(r_2, \tau_2)$ and $(r_1, \tau_1)$ respectively.
\end{proposition}

This theorem has the following simple analytic form
\begin{equation}
\tan(\frac{\psi_3 - \psi_0}{4}) = \frac{\lambda_1 + \lambda_2}{\lambda_1 - \lambda_2}
\tan (\frac{\psi_1 - \psi_2}{4})
\end{equation}
which does not involve differentiation.  The constants $\lambda_1$ and $\lambda_2$ are
values of the parameter $\lambda$ used in equations (3.6) to generate $\psi_1$ and
$\psi_2$ from $\psi_0$ respectively.  Equation (3.7) is called the \emph{nonlinear
superposition formula} for sine-Gordon equation.  It can be used to generate the so
called N-soliton solutions.

 Some other examples of geometric transformations are given below.
\begin{example}

Transformation $W$.  If surfaces $M$ and $M^\ast$ in 3-space have a one to one
correspondence between their points $P$ and $P^\ast$ such that:  (i) $PP^\ast$ is
tangent to $M$ and $M^\ast$ at $P$ and $P^\ast$ respectively, (ii) The asymptotic nets
of $M$ and $M^\ast$ are in correspondence.  Then we say that $M$ and $M^\ast$ are
related by a transformation $W$.  The two parameter family of lines $PP^\ast$ is
called a $W$ congruence of lines.

The above definition is projective in nature and therefore applies to Euclidean and
affine surfaces as well.  The classical B\"{a}cklund transformation given earlier is a
special case of transformation $W$.  Permutability theorem similar to proposition 3.3
for such transformations was also discovered by Bianchi.
\end{example}
\begin{example}
Chern-Terng's transformation.  Chern-Terng [9] studied a pair of affine surfaces
related by a transformation $W$ which satisfies a further condition:  the affine
normals at the corresponding points of the two surfaces are always parallel to each
other.  They showed that such transformations exist if and only if the given surfaces
are affine minimal surfaces.
\end{example}
\begin{example}
A line congruence is a two parameter family of lines, namely a 2-submanifolds of a
certain Grassmannian.  Many examples of transformations of line congruences can be
found in Finikov [13].
\end{example}

\section{General concepts of geometric transformations}

It seems that all the known examples of geometric transformations are among submanifolds in homogeneous spaces.  We will outline Cartan's moving frame method in this general context (cf. Cartan [2], Chern [5]) as the bases of the present approach.\\
\indent \textbf{1.  Cartan's general concept of moving frames.}  Consider space $E$
transformed transitively by a Lie group $G$ on the \emph{right}.  Let $C$ be a
geometric configuration in $E$ (such as a finite collection of points).  As $G$
transforms the space $E, C$ is transformed $C$ accordingly by $C \mapsto Cg$.  It is
easy to find a $C$ without self symmetry under $G$, which means $C g_1 \neq C g_2$
when $g_1 \neq g_2$.  Then the space $\mathcal{F} = \{F|F = Cg$ for $g \in G\}$ serves
as a complete system of reference frames for the space $E$.  The total frame space
$\mathcal{F}$ is a group space transformed by $G$ on both left and right side by the
formulas
\begin{equation}
(Cg)h = Cgh,\quad h(Cg) =Chg.
\end{equation}
The right transformation reflects the original transformation of $G$ on $E$.  But the
left transformation on $\mathcal{F}$ is a more subtle kinds of symmetry of $E$.  Its
effect on submanifold geometry seems to be one of the key to geometric transformation
theory.

We can choose a point $x_0 \in E$ as the origin of $C$.  Then each frame $F=Cg$ has an
origin $x_0g$.  It can be easily verified that two frames $F$ and $F^\prime$ have the
same origin if and only if $F^\prime = h_0F$ for an element $h_0$ in the isotropy
group $H_0 \subset G$ at $x_0$.

For a $p$-dimensional submanifold $M \subset E$, all frames whose origin lies on $M$
form the space $\mathcal{F}_0 \subset \mathcal{F}$ of zeroth order frames of $M$.  By
fixing a particular relative position between a zeroth order frame at $x \in M$ and
the tangent plane of $M$ at $x$, we can define the subspace $\mathcal{F}_1 \subset
\mathcal{F}_0$ of first order frames of $M$.  The second order contact of $M$ can be
used to define the second order frames.  In fact, for any integer $s \geq 0$, one can
define the space $\mathcal{F}_s \subset \mathcal{F}$ of $s$-th order frames of $M$.
Furthermore there is a subgroup $H_s$ of $G$ such that in generic cases, any two
$s$-th order frames at the same point $x \in M$ differ by a left multiplication of an
element in $H_s$.  Therefore a change of frame is given by $F_s \longmapsto h_s F_s$
with $h_s \in H_s$.

In applications, the higher order frames are constructed inductively by the calculus
of Maurer-Cartan forms.  For simplicity we realize $G$ as a matrix groups.  Then all
group elements $T \in G$ and frames $F \in \mathcal{F}$ are matrices.  The total
Maurer-Cartan form $\omega$ is defined by the matrix equation $dF = \omega F$.  All
the entries of $\omega$ are right invariant forms on $\mathcal{F}$.  We can choose a
complete linear independent set of these forms and call each of them a Maurer-Cartan
form.  The total number of them is equal to the dimension of the group $G$.  Further
differentiation and $d^2F =0$ leads to the structure equation $d\omega = \omega \wedge
\omega$ of the group $G$.

It turns out that the $s$-th order frame space $\mathcal{F}_s$ is defined by a set of
equations among the Maurer-Cartan forms which we write simply as $L_s (\omega)=0$.  We
thus have the structure equations
\begin{equation}
\begin{cases}
d \omega = \omega \wedge \omega,\\
L_s (\omega) =0
\end{cases}
\end{equation}
for the $s$-th order frames of $p$-dimensional submanifold in $E$.  Equations (4.2) can be viewed as a differential system on $\mathcal{F}$ with the following property:  for any $p$-dimensional submanifold $M \subset E$, the set of its $s$-th order frames as a submanifold of $\mathcal{F}$ is an integral manifold of (4.2).  Conversely, any integral manifold of (4.2) satisfying appropriate independence condition is the space of $s$-th order frames of some $p$-dimensional submanifold $N \subset E$.  \\
\indent \textbf{2.  Geometric transformations.}  Motivated by concrete examples of
geometric transformations such those given in the previous section, we give
\begin{definition}
A geometric transformation (of order $s$) between two submanifolds $M$ and $M^\ast$ in
$E$ is a one-to-one correspondence between them which satisfies a system of
differential equations (with order less or equal to $s$) invariant under the
transformation of $G$.
\end{definition}

From moving frame point of view, we state the following\\
\\
\textbf{Conjecture}  \emph{The system of differential equations in definition 4.1 is
equivalent to a system of algebraic equations among some algebraic expressions of the
Maurer-Cartan forms of the $s$-th order frames $F_s$ and $F^\ast_s$ of submanifolds
$M$ and $M^\ast$.  Furthermore, the system is invariant under the changes of frames
$F_s \mapsto h_s F_s$ and $F^\ast_s \mapsto h^\ast_s F^\ast_s$.}

Note that above we use the word algebraic in a broad sense, without pursuing a rigorous approach.

An interesting special case of definition 4.1 is
\begin{definition}
A $s$-th order geometric transformation of contact type between two submanifolds $M$
and $M^\ast$ in $E$ is a one to one correspondence between them such that a moving
frames relation of the form $F^\ast_s = T F_s$ holds, where $T$ belongs to a subset $W
\subset G$ satisfying $H_s W H_s = W$.
\end{definition}

With some smoothness assumption for $W$, we can differentiate the equation $F^\ast_s =
T F_s$ and get $\omega^\ast = T \omega T^{-1} + dTT^{-1}$.  Together with the
structure equations of $M$ and $M^\ast$, we have
\begin{equation}
\begin{cases}
d \omega = \omega \wedge \omega, \quad L_s (\omega) =0,\\
d \omega^\ast = \omega^\ast \wedge \omega^\ast, \quad L_s (\omega^\ast)=0,\\
\omega = T^{-1} \omega^\ast T - T^{-1} dT,\\
\omega^\ast = T \omega T^{-1} + dTT^{-1}.
\end{cases}
\end{equation}
Remark that $F^\ast_s = TF_s$ is invariant under the change of frames, and therefore the
System (4.3) may be called {\it the structure equations of the geometric transformation} $F^\ast_s = TF_s$.  
Note that the last two equations of (4.3) are algebraically equivalent, but we write both of them out only
for the sake of symmetry.

As a consequence, we state and prove the following
\begin{proposition}
System (4.3) is equivalent to
\begin{equation}
\begin{cases}
d \omega = \omega \wedge \omega, \quad L_s (\omega) =0,\\
L_s (T \omega T^{-1} + d TT^{-1}) =0,
\end{cases}
\end{equation}
where $T \in W$.  In other words we have completely eliminated $\omega^\ast$ from the
system (4.3).
\end{proposition}
\begin{proof}
We can define $\omega^\ast = T \omega T^{-1} + dTT^{-1}$ and verify $d\omega^\ast =
\omega^\ast \wedge \omega^\ast$ by direct computation.  The last equation of (4.4)
gives $L_s (\omega^\ast)=0$.  We thus all the equations of (4.3).
\end{proof}
The construction of geometric transformations $F^\ast = TF$ is essentially reduced to
the solution of differential system (4.3) or (4.4).  For a given $M$, the first two
equations of (4.4) are automatically satisfied.  The solvability of the third equation
$L_s(T \omega T^{-1} + d TT^{-1}) = 0$ gives the conditions for $M$ to permit
geometric transformation of the given type.  These conditions may further be
translated into conditions on the differential invariants of $M$.  This is a general
formulation of the principle underlying the proof of proposition 3.1 and proposition
3.2 as outlined in \S 3, which will also be the bases of \S 5.

The freedom of frame change $F_s \longmapsto h_s F_s$ and $F^\ast_s \longmapsto
h^\ast_s F^\ast_s$ gives $T \longmapsto h^\ast_s T h^{-1}_s$ which can be used to
reduce the set $W$.  If this transformation is transitive on $W$, we can reduce $T$ to
a constant $T_0$ and get $F^\ast_s = T_0 F_s$.  One such case is the classical
B\"{a}cklund transformation.

The next section is a complete analysis of first order transformations with constant $T$ for surfaces in $E^3$.

\section{Transformations of linear Weingarten Surfaces}

In the spirit of Chern-Terng [9], this section studies first order geometric
transformations of contact type between two surfaces in $E^3$ with constant $T$.  It
leads naturally to the class of surfaces satisfying a relation of the form $aK + 2bH
+c =0$ with constant coefficients.  They are called Linear Weingarten surfaces by some
authors (cf.[1]).

Let $F = \{x, e_1, e_2, e_3\}$ and $F^\ast = \{x^\ast, e^\ast_1, e^\ast_2, e^\ast_3\}$
denote the first order frames of surfaces $M$ and $M^\ast$ in $E^3$ respectively.  A
transformation between the two surfaces defined by $T$ is a one to one correspondence
between them given by $x = x(u,v)$ and $x^\ast = x^\ast (u,v)$, together with two
frame fields $F = F(u,v)$ and $F^\ast = F^\ast (u,v)$ such that $F^\ast (u,v) = TF
(u,v)$ holds for every $(u,v)$.  This can be written explicitly as

\begin{equation}
\begin{cases}
x^\ast = x + b_1 e_1 + b_2 e_2 + b_3 e_3,\\
e^\ast_1 = a_{11} e_1 + a_{12} e_2 + a_{13} e_3,\\
e^\ast_2 = a_{21} e_1 + a_{22} e_2 + a_{23} e_3,\\
e^\ast_3 = a_{31} e_1 + a_{32} e_2 + a_{33} e_3,\\
\end{cases}
\end{equation}
with constant coefficients $b_i$ and constant orthogonal matrix $(a_{ij})$.  We would
like to know what can be said about $M$ and $M^\ast$ from (5.1).  By proposition 4.1,
the existence of transformation (5.1) is equivalent to the solvability of the system
\begin{equation}
\begin{cases}
d \omega = \omega \wedge \omega, \quad L_1 (\omega)=0,\\
L_1 (\omega^\ast) = L_1 (T \omega T^{-1}) =0.\\
\end{cases}
\end{equation}
For given $M$ the first two equations of the above system are automatically satisfied.
So we need only to analyze the solvability of the third equation.  In our case it is
clear that $L_1 (\omega^\ast) = \omega^\ast_3 = dx^\ast \cdot e^\ast_3$.  From (5.1)
we get
\begin{equation}
\omega^\ast_3 = a_{31} \omega_1 + a_{32} \omega_2 + \vline
\begin{aligned}
&b_1 \, &b_2\\
&a_{31} \, &a_{32}
\end{aligned}
\vline\, \omega_{12} + \vline
\begin{aligned}
& b_1 \, & b_3\\
&a_{31} \, &a_{33}
\end{aligned}
\vline\, \omega_{13} + \vline
\begin{aligned}
& b_2 \, & b_3\\
& a_{32} \, &a_{33}
\end{aligned}
\vline\, \omega_{23}.
\end{equation}
We therefore need to study the solvability of
\begin{equation}
p_1 \omega_1 + p_2 \omega_2 + p_{12} \omega_{12} + p_{13} \omega_{13} + p_{23}
\omega_{23} =0
\end{equation}
with
\begin{equation}
\begin{cases}
p_1 = a_{31}, \quad p_2 = a_{32},\\
p_{12} = \vline
\begin{aligned}
& b_1 & b_2\\
& a_{31} & a_{32}
\end{aligned}
\vline\,, \quad p_{13} =\, \vline
\begin{aligned}
& b_1 & b_3\\
&a_{31} & a_{33}
\end{aligned}
\vline\,, \quad p_{23} =\, \vline
\begin{aligned}
& b_2 & b_3\\
&a_{32} & a_{33}
\end{aligned}
\vline\, .
\end{cases}
\end{equation}
We further assume that the four points $x, x^\ast, x + e_3, x^\ast + e^\ast_3$ are not
in the same plane, in other words $xx^\ast \wedge e_3 \wedge e^\ast_3 \neq 0$.  From
(5.1) we have
\begin{equation}
xx^\ast \wedge e_3 \wedge e^\ast_3 = - \vline
\begin{aligned}
&b_1 & b_2\\
&a_{31} & a_{32}
\end{aligned}
\vline \,e_1 \wedge e_2 \wedge e_3 = -p_{12} e_1 \wedge e_2 \wedge e_3.
\end{equation}
We thus have $p_{12} \neq 0$.
We now study the solvability of (5.4) for a given surface $M$.  More precisely we ask
under what condition can we find a frame field $F (u,v)$ of $M$ along which (5.4)
holds.  For clarity let's define 1-form $\alpha = p_1 \omega_1 + p_2 \omega_2 + p_{12}
\omega_{12} + p_{13} \omega_{13}
 + p_{23} \omega_{23}$ on the 3-dimensional frame space $\mathcal{F}_1$.  We generally use the same letter to denote differential forms on a manifold and their restrictions to submanifolds.  Somethimes, we need to be very careful about the differences of meanings.  But we still prefer not to stress these differences by more notations.

If $\alpha =0$ holds along the frame field $F(u,v)$, then $d\alpha=0$ must also hold.
In order to analyze the consequences of these two equations, lets first consider some
algebraic identities on $\mathcal{F}_1$.  By the first 5 equations of (5.4),
\begin{equation}
d \alpha = (p_2 \omega_1 - p_1 \omega_2 + p_{23} \omega_{13} - p_{13} \omega_{23})
\wedge \omega_{12} - p_{12} \omega_{13} \wedge \omega_{23}.
\end{equation}
Since $\omega_1, \omega_2, \omega_{12}$ form a basis of 1-forms at every point $p_{12}
\neq 0, \omega_1, \omega_2, \alpha$ also form such a basis at every point.  We can
define function $h$ uniquely by $d \alpha = h \omega_1 \wedge \omega_2 \pmod{\alpha}$,
or equivalently by $d\alpha \wedge \alpha = p_{12} h \omega_1 \wedge \omega_2 \wedge
\omega_{12}$.  From (5.7) and
\begin{equation}
\begin{cases}
\omega_{13} \wedge \omega_{23} = K \omega_1 \wedge \omega_2,\\
\omega_1 \wedge \omega_{23} - \omega_2 \wedge \omega_{13} = 2 H \omega_1 \wedge
\omega_2,
\end{cases}
\end{equation}
we get
\begin{equation}
d \alpha \wedge \alpha = -[(p^2_{12} + p^2_{13} + p^2_{23}) K + 2 (p_1 p_{13} + p_2
p_{23} ) H + (p^2_1 + p^2_2)] \omega_1 \wedge \omega_2 \wedge \omega_{12}
\end{equation}
which determines $h$.  We therefore have
\begin{equation}
d \alpha = -\frac{1}{p_{12}} [(p^2_{12} + p^2_{13} + p^2_{23}) K + 2 (p_1 p_{13} + p_2
p_{23} ) H + (p^2_1 + p^2_2)] \omega_1 \wedge \omega_2 \pmod{\alpha}
\end{equation}
which is true on the whole of $\mathcal{F}_1$.  Restricting to the frame field
$F(u,v)$ along which $\alpha= d \alpha =0$, we get
\begin{equation}
(p^2_{12} + p^2_{13} + p^2_{23}) K + 2 (p_1 p_{13} + p_2 p_{23}) H + (p^2_1 +
p^2_2)=0.
\end{equation}
The left hand side of the above equation depends only on the points $x \in M$.  The
equation is therefore frame independent.  Conversely, if a surface $M$ satisfy (5.11),
then by (5.10)
\begin{equation}
d \alpha =0\, \pmod{\alpha}
\end{equation}
holds on the whole space of first order frames of $M$.  We can then solve the system
$\alpha=0$ uniquely for any initial condition by Frobenius theorem.  We can now
summarize all our conclusion as
\begin{proposition}
For a given surface $M$, the existence of a transformation (5.1) with $b_1 a_{32} -
b_2 a_{31} \neq 0$ is equivalent to the linear Weingarten relation
\begin{equation}
(p^2_{12} + p^2_{13} + p^2_{23}) K + 2 (p_1 p_{13} + p_2 p_{23}) H + (p^2_1 + p^2_2)=0
\end{equation}
where the constants $p_1, p_2, p_{12}, p_{13}, p_{13}$ are determined by coefficients
of $T$ through (5.5).  Moreover if a fixed $T$ defines at least one transformation
(5.1), then it defines a unique such transformation for a generic initial frame $F$ at
any point of $M$.  This gives a one parameter family of transformations of $M$.
\end{proposition}
\begin{remark}
Generic (rather than arbitrary) initial frame is used above because we want to avoid
the degenerate cases where the independence condition $\omega^\ast_1 \wedge
\omega^\ast_2 \neq 0$ does not hold.

For further analysis we let
\begin{equation}
\begin{cases}
p^2_{12} + p^2_{13} + p^2_{23} = a,\\
p_1 p_{13} + p_2 p_{23} =b,\\
p^2_1 + p^2_2 =c,
\end{cases}
\end{equation}
and write (5.13) as $aK + 2bH + c =0$.  The assumption $p_{12} = b_1 a_{32} - b_2
a_{31} \neq 0$ implies that $c = p^2_1 + p^2_2 = a^2_{31} + a^2_{32} > 0$.  So we have
$a > 0$ and $1 > c > 0$ and $b^2 - ac < 0$.  In fact we can show that $b^2 - ac =
-p^2_{12}$ with some computation.  Since $aK + 2bH + c =0$ is equivalent to $taK +
2tbH + tc =0$ for nonzero constant $t$, in some sense $a, b, c$ are only related by
$b^2 - ac < 0$.

Now let's begin with an arbitrary $M$ satisfying $aK+2bH+c =0$ with $b^2 -ac < 0$.  We
can assume both $a$ and $c$ to be positive and find a set of constants $p_1, p_2,
p_{12}, p_{13}, p_{23}$ with $p_{12} \neq 0$ such that (5.14) holds.  We like to see
if we can construct a $T$ whose coefficients satisfy (5.5).  Since we can multiply the
constants $a,b,c$ by an arbitrary common factor $s^2 \neq 0$, the equation (5.5) can
be rewritten as
\begin{equation}
\begin{cases}
(a_{31}, a_{32}) = (sp_1, sp_2),\\
(b_1, b_2, b_3) \times (a_{31}, a_{32}, a_{33})= (sp_{23}, -sp_{13}, sp_{12}).
\end{cases}
\end{equation}
Letting $a_{33} = sp_3$, the second equation above becomes
\begin{equation*}
(b_1, b_2, b_3) \times (p_1, p_2, p_3) = (p_{23}, -p_{13}, p_{12}).
\end{equation*}
This requires $(p_1,p_2, p_3) \cdot (p_{23}, -p_{13}, p_{12}) =0$ and $p_3 = -(p_1
p_{23} - p_2 p_{13})/p_{12}$.  Because $b^2 - ac < 0$, we have $p^2_1 + p^2_2 = c \neq
0$ and therefore $(p_1, p_2, p_3) \neq 0$.  This shows that the required $(b_1, b_2,
b_3)$ can be found.  Now we can use
\begin{equation*}
1 = a^2_{31} + a^2_{32} + a^2_{33} = s^2 (p^2_1 + p^2_2 + p^2_3) = s^2 (p^2_1 + p^2_2
+ (p_1 p_{23} - p_2 p_{13})^2 /p^2_{12})
\end{equation*}
to determine $s$ and therefore $a_{31}, a_{32}, a_{33}$.  We have thus solved (5.5)
and found the required $T$.  But for a set of constants $a,b,c,$ the matrix $T$ is not
unique.  The exact relation will be clarified in proposition 5.3.  Now we can restate
part of proposition 5.1 as
\end{remark}
\begin{proposition}
The condition for a surface $M$ to permit a first order transformation (5.1) with
$xx^\ast \wedge e_3 \wedge e^\ast_3 \neq 0$ (or equivalently $b_1 a_{32} - b_2 a_{31}
\neq 0)$ is that is satisfies a linear Weingarten relation
\begin{equation}
aK + 2bH + c =0
\end{equation}
with $b^2 - ac < 0$.
\end{proposition}
\begin{remark}
The condition $b^2 - ac< 0$ is caused by our restriction to real geometry.  As we
progress further into the algebraic aspects of the theory, it becomes natural to work
in the complex field.  More precisely, we can complexify the space $E^3$ and consider
its analytic surface theory.  With this extension, surfaces with $b^2 - ac>0$ permit
complex transformations.

We would like to give a more geometric description of the relation (5.1).  We are only
concerned with the relative position between the two tangent planes of $M$ and
$M^\ast$ at $x$ and $x^\ast$, which is not affected by rotations of the frames $F$ and
$F^\ast$ around $e_3$ and $e^\ast_3$.  This relative position can be described by the
following geometric quantities:
\begin{equation}
\begin{cases}
(i) \text{the distance}\; r = ||xx^\ast||,\\
(ii) \text{the angle}\; \tau\; \text{between}\; e_3\; \text{and}\; e^\ast_3,\\
(iii) \text{the angle}\; \theta\; \text{between the cord}\; xx^\ast \; \text{and}\; e_1 e_2- \text{plane},\\
(iv) \text{the angle}\; \theta^\ast \; \text{between the cord}\; xx^\ast \text{and} \;
e^\ast_1 e^\ast_2-\text{plane}.
\end{cases}
\end{equation}
Their relations to the coefficients of $T$ can be obtained as follow,
\begin{equation}
\begin{cases}
r^2 = ||xx^\ast||^2 = b^2_1 + b^2_2 + b^2_3,\\
\cos \tau = e_3 \cdot e^\ast_3 = a_{33},\\
\sin \theta = xx^\ast \cdot e_3/||xx^\ast|| = b_3/r,\\
-\sin \theta^\ast = xx^\ast \cdot e^\ast_3/||xx^\ast|| = (a_{31} b_1 + a_{32} b_2 +
a_{33} b_3)/r.
\end{cases}
\end{equation}
The four constants $r^2, \cos \tau, \sin \theta, \sin \theta^\ast$ represent an
equivalent class of $T$ under the relation $T \sim h_1 T h^\ast_1$, where $h_1$ and
$h^\ast_1$ belongs to the subgroup $H_1$ consisting of transformations of the form
\begin{equation}
\begin{cases}
\overline{x} = x,\\
\overline{e}_1 = \cos \phi e_1 + \sin \phi e_2,\\
\overline{e} _2 = -\sin \phi e_1 + \cos \phi e_2,\\
\overline{e} _3 = e_3.
\end{cases}
\end{equation}
By equation (5.5), the coefficients of (5.11) or (5.16) can be expressed in terms of
the quantities in (5.17) as follow,
\begin{equation*}
a = p^2_{12} + p^2_{13} + p^2_{23} = ||xx^\ast \times e^\ast_3||^2 = r^2 \cos^2
\theta^\ast,
\end{equation*}
\begin{equation*}
c = p^2_1 + p^2_2 = a^2_{31} + a^2_{32} =1 - a^2_{33} =1 - \cos^2 \tau =\sin^2 \tau,
\end{equation*}
and
\begin{align*}
b = &p_1 p_{13} + p_2 p_{23} = a_{31} \vline
\begin{aligned}
&b_1 & b_3\\
&a_{31} & a_{33}
\end{aligned}
\vline + a_{32} \vline
\begin{aligned}
&b_2 & b_3\\
&a_{32} & a_{33}
\end{aligned}
\vline\\
& = (a_{31} b_1 + a_{32} b_2) a_{33} - (a^2_{31} + a^2_{32}) b_3\\
& = (-r \sin \theta^\ast - \cos \tau (r \sin \theta)) \cos \tau - \sin^2 \tau (r \sin \theta)\\
& = -r (\sin \theta^\ast \cos \tau + \sin \theta).
\end{align*}
Equation (5.16) can now be written as
\begin{equation}
r^2 \cos^2 \theta^\ast K - 2r(\cos \tau \sin \theta^\ast + \sin \theta) H + \sin^2
\tau =0.
\end{equation}
Similarly we have
\begin{align*}
p^2_{12} & = (b_1 a_{32} - b_2 a_{31})^2\\
& = (b^2_1 + b^2_2)(a^2_{31} + a^2_{32}) - (b_1 a_{31} + b_2 a_{32})^2\\
& = r^2 \cos^2 \theta \sin^2 \tau + r^2 (\sin \theta^\ast + \cos \tau \sin \theta)^2\\
& = r^2(\sin(\tau + \theta) + \sin \theta^\ast)(\sin (\tau - \theta) - \sin \theta^\ast)\\
& = 4 r^2 \sin \frac{\tau + \theta + \theta^\ast}{2} \cos \frac{\tau + \theta - \theta^\ast}{2} \sin \frac{\tau - \theta - \theta^\ast}{2} \cos \frac{\tau - \theta + \theta^\ast}{2}\\
& = r^2 (\cos (\theta + \theta^\ast) - \cos \tau) (\cos (\theta - \theta^\ast) + \cos
\tau).
\end{align*}
The condition $p_{12} \neq 0$ now reads $r \neq 0$ and $\cos (\theta \pm \theta^\ast)
\neq \pm \cos \tau$.  Since $M$ and $M^\ast$ are in symmetric positions, we have the
following
\end{remark}
\begin{proposition}
If surfaces $M$ and $M^\ast$ are related by a first order transformation geometrically
given by (5.17) with $xx^\ast \wedge e_3 \wedge e^\ast_3 \neq 0$ (or equivalently $r
\neq 0$ and $\cos (\theta \pm \theta^\ast) \neq \pm \cos \tau)$, then they are both
linear Weingarten surfaces satisfying
\begin{equation}
\begin{cases}
r^2 \cos^2 \theta^\ast K - 2r (\cos \tau \sin \theta^\ast + \sin \theta) H + \sin^2 \tau =0,\\
r^2 \cos^2 \theta K^\ast - 2r (\cos \tau \sin \theta + \sin \theta^\ast ) H^\ast +
\sin^2 \tau =0.
\end{cases}
\end{equation}
\end{proposition}

Using the notation
\begin{equation}
\begin{cases}
a = r^2 \cos^2 \theta^\ast, \;b = -r(\cos \tau \sin \theta^\ast + \sin \theta), \;c = \sin^2 \tau,\\
a^\ast = r^2 \cos^2 \theta, \;b^\ast = -r(\cos \tau \sin \theta + \sin \theta^\ast),
\;c^\ast = \sin^2 \tau,
\end{cases}
\end{equation}
we can verify that $c = c^\ast$ and $ac - b^2 = a^\ast c^\ast - b^{\ast2}$.  But we
should allow $a,b,c$ and $a^\ast, b^\ast, c^\ast$ to be changeable upto arbitrary
common factors $t$ and $t^\ast$ respectively.  Therefore they only have the relation
\begin{equation}
(ac - b^2)/c^2 = (a^\ast c^\ast - b^{\ast2})/c^{\ast2}.
\end{equation}
The number $(ac-b^2)/c^2$ is invariant under the first order transformations.

To complete our analysis, we briefly look at the independence condition $\omega^\ast_1
\wedge \omega^\ast_2 \neq 0$ which insures that $M^\ast$ is an immersed surface in
$E^3$.  Similar to (5.3) we can express $\omega^\ast_1$ and $\omega^\ast_2$ in terms
of $\omega_i$ and $\omega_{ij}$.  Then by (5.4) and (2.6), we can obtain
\begin{equation}
\omega^\ast_1 \wedge \omega^\ast_2 = \frac{1}{p_{12}} Q(T, h_{11}, h_{12}, h_{22})
\omega_1 \wedge \omega_2.
\end{equation}
Here $Q$ is a polynomial which is quadratic in $h_{11}, h_{12}$ and $h_{22}$.

\section{ Linear Weingarten surfaces with $aH + 2 b K = 0$ (case $c=0$) and their (nonlinear Weingarten) transformations}

Let $F = \{x, e_1, e_2, e_3\}$ and $F^\ast = \{x^\ast, e^\ast_1, e^\ast_2, e^\ast_3\}$
denote the first order frames of surfaces $M$ and $M^\ast$ in $E^3$ respectively.  Assume the transformation between the two frames is given by rotation in plane, so that $e_3$ and $e^\ast_3$ are either identical, or opposite normals. Assume that such a transformation $F^\ast (u,v) = TF
(u,v)$ holds for every $(u,v)$ on the original surface.  If the Gauss map changes orientation, we say that the {\it tangent planes are identical up to a change in orientation}. An example of such a transformation is given by the equations:

\begin{equation}
\begin{cases}
x^\ast = x + b_1 e_1 + b_2 e_2 + b_3 e_3,\\
e^\ast_1 = a_{11} e_1 + a_{12} e_2,\\
e^\ast_2 = a_{21} e_1 + a_{22} e_2,\\
e^\ast_3 = - e_3,\\
\end{cases}
\end{equation}
with constant coefficients $b_i$ and constant orthogonal matrix $(a_{ij})$.

Inspired by the previous idea, we would like to remind the concept of Christoffel transform.

Two immersions $M$ and $M^*$ from the same open simply connected domain $D$ into Euclidean space are said to be Christoffel transforms of each other (a Christoffel pair) if they induce two conformally equivalent metrics on their (same) domain, such that, at every point, their tangent planes are parallel, and with opposite orientations.

\

We will reserve the term of {\bf dual} surface for a Christoffel transform that is made unique in a prescribed way (see definition below). Clearly, a Christoffel pair is uniquely defined up to homotheties and translations.

\

 Such a corresponding surface $M^\ast$ of $M$, if it exists, is said to be a Christoffel transform (or generalized-dual) of the first one. Note that a Christoffel transform does not exist for every surface. Actually, Christoffel pairs are characterized by the existence of isothermic parameterizations.

\

 An isothermic parameterization around each (non-umbilic) point of a surface is a conformal (isothermal) parameterization which also diagonalizes the second fundamental form.
 Another way to characterize it is: a conformal parameterization for which the Hopf differential is real-valued.

 A surface that admits an isothermic parameterization is said to be {\bf isothermic}.

 For example, any constant mean curvature surface admits an isothermic parameterization away from umbilics.
 However, note that even the trivial case of totally umbilic surfaces can be included in the general family of isothermic surfaces.

\

{\bf Question 1:} Do linear Weingarten surfaces admit isothermic parameterizations? The answer is negative in general.

Examples of isothermic surfaces include constant mean curvature surfaces, Bonnet surfaces, quadric surfaces. 
While a theorem of characterization for isothermic surfaces exists, all (linear) Weingarten surfaces have not yet been classified from the view point of isothermic parameterizations.

\

{\bf Question 2:} Is the Finkel-Wu's conjecture true? This conjecture basically states that the only functional relation of the type $G(H,K)=0$ (or, equivalently, the only relation between principal curvatures) which determines an Integrable Class of Weingarten surfaces is the {\it linear} relation.
Answer: Recently, Baran and Marvan brought some excellent arguments disproving (rephrasing) that conjecture.

\

{\bf Question 3:} Let us consider the particular case of a linear Weingarten surface $M$ for which $c=0$.
If the surface $M^\ast$ is obtained from $M$ via a first order transformation such as the one seen before, will the new surface $M^\ast$ also represent a linear Weingarten surface? If not, what is the geometric relationship between these two surfaces, and how can it it be characterized?

\begin{remark}
Propositions 5.2 and 5.3 do {\bf not} apply to this special case. Indeed, in proposition 5.3 we have excluded the case when the angle between the vectors $e_3$ and $e^\ast_3$ is either zero, or a multiple of $\pi$. Note that if we took the angle $\tau$ to be either zero, or a multiple of $\pi$, as a Limiting Case of Proposition 5.3, we could obtain a sequence of linear Weingarten surfaces $aH+2bK+c_n=0$ with $c_n$ converging to zero.
\end{remark}

 In order to provide an answer to question Q3 formulated before, let us first consider a Weingarten surface $f$ that satisfies the relation $aH+2bK=0$ with constant real coefficients $a$ and $b$, and assume that such a surface admits isothermic coordinates $(u,v)$. With respect to these coordinates, let its induced metric be $ds^2 = E (du^2 + dv^2)$ and the second fundamental form be $d \sigma^2= l du^2 + n dv^2$.

\begin{definition}
Given an isothermic immersion $f(u,v)$ into Euclidean space, with metric $ds^2 = E (du^2 + dv^2)$, there exists a specialized Christoffel transform (called {\bf dual}) of $f$, denoted $f^\ast$, such that $(f^\ast)_u={f_u}/E$ and $(f^\ast)_v=-{f_v}/E$. The metric factor of $f^\ast$ represents the inverse of the metric factor of $f$, namely $ds^2 = E^{-1} (du^2 + dv^2)$ , while its Gauss map satisfies $N^\ast= - N$ everywhere defined.
\end{definition}

The proof and additional properties are straightforward and can be also found in reference [17].

Note that the second fundamental form of the dual will be:
$$d \sigma^2= E^{-1} ( - l  du^2 + n  dv^2)$$.

\

Consequently, the principal curvatures of the dual surface will be:

$k^\ast_1 = -l = - E k_1$, and, respectively
$k^\ast_2 = n = E k_2$.

\begin{remark}
 Consider a linear Weingarten surface with non-vanishing Gaussian curvature which satisfies the equation $H = \alpha K$, where $\alpha$ is a real constant. An equivalent way to characterize this surface is $\frac{1}{k_1} + \frac{1}{k_2} = constant$, where $k_1$ and $k_2$ are the principal curvatures.
\end{remark}

Further, we state the following:

\begin{proposition}
 If an isothermic linear Weingarten surface $f$ satisfies the equation $H = \alpha K$, away from points where $K$ vanishes, then its {\it dual surface} $f^\ast$  satisfies the relation $\frac{1}{k^\ast_1} - \frac{1}{k^\ast_2} = constant$, that is,  represents a surface with constant astigmatism. For more details about this type of surface, see [10].
\end{proposition}

\begin{proof}
The proof is immediate and it is based on the previous remark and proposition for the dual surface of an isothermic immersion.
\end{proof}

Note that the dual surface $f^\ast$ can be actually obtained from the original surface $f$ by applying a first order transformation (Euclidean motions of the corresponding moving frame).

On the other hand, a linearity relation of type $aH+2bK=0$ is {\bf not} preserved by first order transformations.

\

However, proposition 6.1 proved a beautiful duality between the relations satisfied by the principal curvatures (of the original and the dual, respectively: $\frac{1}{k_1} + \frac{1}{k_2} = constant$ versus $\frac{1}{k^\ast_1} - \frac{1}{k^\ast_2} = constant$ ).
Also note that the {\it special} case of any minimal surfaces $H=0$ and its dual sphere trivially satisfies the proposition above. The dual surface of a minimal surface is a sphere in the strict sense of the previous definition, and in particular, the sphere has a constant astigmatism.

\

Surfaces with constant astigmatism, defined by $\frac{1}{k^\ast_1} -\frac{1}{k^\ast_2} = constant$, {\bf do} represent Weingarten surfaces, in the sense that there exists a relationship between $H$ and $K$, but this relation is {\bf not} linear.

The paper [10] showed and rigorously proved that all surfaces with constant astigmatism represent involutes of pseudospherical surfaces (i.e., surfaces of constant negative Gauss curvature represent the evolutes of surfaces with constant astigmatism).

The goal of [10] was to find a meaningful relationship between the class of surfaces with constant astigmatism and the class of linear Weingarten surfaces.

The goal of the current section was to describe yet another relation between surfaces with constant astigmatism and linear Weingarten surfaces.

\

\textbf{Acknowledgement.}  The author would like to thank X. Mo for the private communication that represented the origin of this work. The academic world of mathematics would have greatly benefited from his continuous expertise.

\end{document}